\documentclass[12pt,reqno]{amsart}

\usepackage[T1]{fontenc}
\usepackage{amsmath}						
\usepackage{amssymb}
\usepackage{amsthm}
\usepackage{amscd}
\usepackage{amsfonts}
\usepackage{mathtools}
\usepackage{stmaryrd}
\usepackage{algorithm, algorithmic}
\usepackage{ wasysym }
\usepackage{caption}
\usepackage{subcaption}
\usepackage{comment}

\usepackage{extarrows}
\usepackage[colorlinks, linktocpage, citecolor = red, linkcolor = blue]{hyperref}
\usepackage{color}
\usepackage{tikz}
\usepackage{tikz-cd}
\usepackage{calrsfs}
\DeclareMathAlphabet{\pazocal}{OMS}{zplm}{m}{n}

\usepackage{fullpage}
\usepackage[shortlabels]{enumitem}

\usepackage{newcent}
\usepackage{fouriernc}

\newtheorem{theorem}{Theorem}[section]
\newtheorem{lemma}[theorem]{Lemma}

\newtheorem{maintheorem}{Theorem}

\newtheorem{proposition}[theorem]{Proposition}

\theoremstyle{definition}

\newtheorem{example}[theorem]{Example}

\newtheorem{manualtheoreminner}{Theorem}
\newenvironment{manualtheorem}[1]{%
  \IfBlankTF{#1}
    {\renewcommand{\themanualtheoreminner}{\unskip}}
    {\renewcommand\themanualtheoreminner{#1}}%
  \manualtheoreminner
}{\endmanualtheoreminner}

\newtheorem{remark}[theorem]{Remark}

\newcommand{\edit}[1]{{#1}}

\newcommand{\A}{\mathbb{A}}
\newcommand{\FF}{\mathbb{F}}

\newcommand{\PP}{\mathbb{P}}
\newcommand{\ZZ}{\mathbb{Z}}

\newcommand{\pB}{\pazocal{B}}

\newcommand{\pO}{\pazocal{O}}

\newcommand{\sQ}{\mathsf{Q}}
\newcommand{\pC}{\pazocal{C}}
\newcommand{\Sn}[1]{\mathfrak{S}_{#1}}

\DeclareMathOperator{\Spec}{Spec}

\DeclareMathOperator{\Aut}{Aut}

\newcommand{\GL}[2]{\operatorname{GL}_{#1}(#2)}

\DeclareMathOperator{\PGL}{PGL}
\DeclareMathOperator{\rowspan}{rowspan}

\newcommand{\simp}{\mathsf{dc}}
\newcommand{\Conf}{\mathsf{Conf}}
\newcommand{\Gr}{\mathsf{Gr}}

\title{The virtual Euler characteristic for binary matroids}

 \author[M. Brandt]{Madeline Brandt}
 \address{Department of Mathematics, Vanderbilt University, TN 37212}
 \email{madeline.v.brandt@vanderbilt.edu}

 \author[J. Bruce]{Juliette Bruce}
 \address{Department of Mathematics, Dartmouth College, Hanover, NH 03755}
 \email{juliette.bruce@dartmouth.edu}

 \author[D. Corey]{Daniel Corey}
 \address{Department of Mathematics, Embry-Riddle Aeronautical University, Daytona Beach, Florida 32114}
 \email{\href{mailto:daniel.corey@erau.edu}{daniel.corey@erau.edu}}

\begin{document}

\maketitle

\begin{abstract}
Inspired by Kontsevich's graphic orbifold Euler characteristic we define a virtual Euler characteristic for any finite set of isomorphism classes of matroids of rank $r$. Our main result provides a formula for the virtual Euler characteristic for the set of isomorphism classes of \edit{simple} matroids of rank $r$ realizable over $\FF_2$.  
We prove this formula by relating the virtual Euler characteristic for \edit{simple} binary matroids to the point counts of certain subsets of Grassmanians over finite fields. 
    
 (secondary)
 
\end{abstract}

\section{Introduction}

As part of his work on graph chain complexes \cite{Kontsevich}, Kontsevich stated the following fascinating formula concerning graphs:
\begin{equation}\label{eq:graph-euler}
  \sum_{G\in \Gamma_{g}} \frac{(-1)^{|V(G)|}}{|\Aut(G)|}=\frac{B_{g}}{g(g-1)}  
\end{equation}
where $B_{g}$ denotes the $g$-th Bernoulli number and $\Gamma_{g}$ is the set of isomorphism classes of connected  genus $g$ graphs whose vertices are at least trivalent. See \cite{borinskyVogtmann22,Gerlits,Penner} for further discussion and proofs of the above formula, and \cite[\S~10]{ChanFarberGalatiusPayne} for an analogous formula for graphs with markings. In this paper we consider the question of whether similar formulas might exist for matroids. More specifically, suppose $\pC(r)$ is a finite set of isomorphism classes of matroids of rank $r$, and let $\pC(r,n)\subset \pC(r)$ be those (isomorphism classes of) matroids in $\pC(r)$ whose ground set has $n$ elements. We define the \emph{virtual Euler characteristic of $\pC(r)$} to be
\begin{equation*}
    \chi(\pC(r)) =  \sum_{n\geq r} \sum_{\sQ\in \pC(r,n)} \frac{(-1)^{n}}{|\Aut(\sQ)|}.
\end{equation*}
Here $\Aut(\sQ)$ denotes the automorphism group of the matroid $\sQ$, which is the group of bijections of the ground set of $\sQ$ preserving the bases of $\sQ$. The extent to which this virtual Euler characteristic estimates the actual Euler characteristic of the corresponding matroid homology chain complex in the sense of \cite{AlekseyevskayaBorovikGelfandWhite} depends on the asymptotic behavior of $|\Aut(\sQ)|$ as the size of the ground set of $\sQ$ grows. 
It is conjectured in \cite{MayhewNewmanWelshWhittle} that asymptotically almost all matroids have a trivial automorphism group, and some progress on this is made in  \cite{PendavinghVanDerPol}; in particular, they show that asymptotically almost all sparse-paving matroids have trivial automorphism group (and it is conjectured in \cite{MayhewNewmanWelshWhittle} that almost all matroids are sparse-paving).
 
There are many natural classes of matroids one might take for $\pC(r)$ in attempting to generalize Equation \eqref{eq:graph-euler}, for instance,  
rank $r$ simple matroids realizable over a fixed finite field, rank $r$ graphic matroids, rank $r$ cographic matroids, and rank $r$ regular matroids. 
In this paper, we focus our attention to the  case of \edit{simple} binary matroids (i.e., matroids realizable over the finite field $\FF_2$). Let $\pB(r)$ be the set of isomorphism classes of simple binary matroids of rank $r$.

\begin{maintheorem}
\label{thm:binaryOEC}
The virtual Euler characteristic for \edit{simple} binary matroids is given by
\begin{equation}\label{eqn:mainThm}
\chi(\pB(r)) = \prod_{i=1}^r\frac{1}{(1-2^i)}.
\end{equation}
\end{maintheorem}

\begin{example}
\label{ex:binary3}
We verify the formula in Theorem \ref{thm:binaryOEC} directly for the case $r = 3$. There are \edit{6} matroids in $\pB(3)$, pictured in Figure \ref{fig:binaryMatroidsRank3}. Denote by $\Sn{n}$ the symmetric group on $\{1,2,\ldots,n\}$.  The automorphism groups of these matroids are, left to right, $\Sn{3}$, $\Sn{4}$, $\Sn{3}$, $D_8$ (the dihedral group with 8 elements), $\Sn{4}$, and $\PGL_3(\FF_2)$, which have orders $6$, $24$, $6$, $8$, $24$, and $168$, respectively. Using this, one calculates $\chi(\pB(3)) = -1/21$, verifying Theorem \ref{thm:binaryOEC} in this case. 
\end{example}

\begin{figure}[h]
    \centering
    \includegraphics[width=\textwidth]{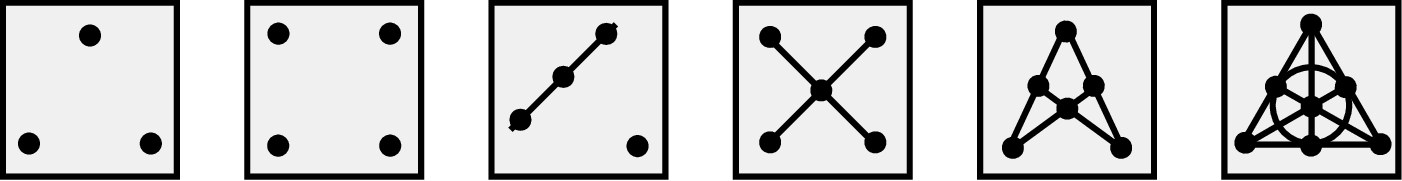}
    \caption{Rank 3 simple binary matroids}
    \label{fig:binaryMatroidsRank3}
\end{figure}

The proof of Theorem~\ref{thm:binaryOEC} revolves around the idea that counting isomorphism classes of \edit{simple} binary matroids can be phrased as a question of counting points of certain subsets of Grassmanians. In particular, the key element of our proof is the following equality
\begin{equation}\label{eq:mainEqualityAutsGrass}
    \sum_{\sQ\in \pB(r,n)} \frac{1}{|\Aut(\sQ)|} = \frac{\left|\Gr^{\simp}(r,n;\FF_2)\right|}{n!}
\end{equation}
where $\Gr^{\simp}(r,n;\FF_q)\subset \Gr(r,n;\FF_q)$ consists of the $r$-dimensional linear subspaces of $(\FF_{q})^n$ represented by $r\times n$ matrices whose columns are distinct and not identically zero.
When $q=2$, this consists of those $r$-dimensional linear spaces in $(\FF_2)^{n}$ realizing a simple binary matroid; see Proposition~\ref{prop:BrnGrsimp}.
Theorem~\ref{thm:binaryOEC} then follows from computing the point count on the right-hand side above.
We do this for $\mathbb{F}_{q}$ where $q$ is a prime power as follows. 

\begin{maintheorem}
\label{thm:grsimp}
We have
\begin{equation*}
\sum_{n \geq r} \frac{(-1)^n}{n!} \left|\Gr^{\simp}(r,n;\FF_q)\right| = \prod_{i=1}^r\frac{1}{(1-q^i)}.
\end{equation*}
\end{maintheorem}

Note that while Equation \ref{eq:mainEqualityAutsGrass} only works for $q=2$, Theorem \ref{thm:grsimp} is for any prime power $q$, see Remark \ref{sec:extendToOtherPrimesP} for a discussion on how to extend Theorem \ref{thm:binaryOEC} to other primes.  

In the course of proving Theorem \ref{thm:grsimp}, we observe the following formula.
\begin{equation}
\label{eq:grothintro}
    \left|\Gr^{\simp}(r+1,n;\FF_q)\right| = 
\sum_{k=r+1}^{n} \left|\Gr^{\simp}(r,k-1;\FF_q)\right|   \prod_{j = k+1}^n (q^{r+1} - j). 
\end{equation}
The summand calculates the the number of $(r+1)\times n$--matrices with $\FF_q$ entries in row-reduced echelon form whose last pivot occurs in the $k$th row. 
We elevate this to a recursive formula 
in Grothendieck ring of varieties.

\begin{maintheorem}
\label{thm:K0}
Given integers $r\geq 1$ and $n\geq r+1$, and an infinite field $F$, we have
\begin{equation}
\label{eq:K0}
    [\Gr^{\simp}(r+1, n; \edit{F})] = \sum_{k=r+1}^{n} [\Gr^{\simp}(r,k-1; \edit{F})] [\Conf_{n-k}(\A_{\edit{F}}^{r+1} \setminus \{k+1 \text{ points}\})].
\end{equation}
\end{maintheorem}
Here, $\Conf_{m}(X)$ the configuration space of \textit{ordered} $m$ distinct points in the variety $X$. Our proof of the above formula requires that the field has a large number of elements. If it also holds for $\FF_q$, then an application of the ring homomorphism 
\begin{equation*}
    K_{0}(\mathrm{Var}_{\FF_q}) \to \ZZ \qquad [X] \mapsto |X(\FF_{q})| 
\end{equation*}
to this formula recovers Formula \ref{eq:grothintro}.

\subsection*{Acknowledgements}
We thank Melody Chan for her thoughtful conversations on graph complexes as well as her feedback throughout the course of this project. We also thank Shiyue Li for helpful conversations during the beginning of this project and for suggesting the potential connection to $\beta$-invariants. \edit{Finally, we thank the anonymous referees for careful reading and thoughtful comments.} The first author is supported by the National Science Foundation under Award No. 2001739. The second author is grateful for the support of the Mathematical Sciences Research Institute in Berkeley, California, where she was in residence for the 2020--2021 academic year. The second author is partially supported by the National Science Foundation under Award Nos. DMS-1440140, NSF FRG DMS-2053221, and NSF MSPRF DMS-2002239. The third author is supported by NSF-RTG grant 1502553 and the SFB-TRR project ``Symbolic Tools in Mathematics and their Application'' (project-ID 286237555).

\section{Counting binary matroids via the Grassmannian}

Throughout we let $[n]\coloneqq \{1,2,\ldots,n\}$ and write $\binom{[n]}{r}$ for the set of subsets of $[n]$ of size $r$. We assume familiarity with matroids, and we point the unfamiliar reader to \cite{Oxley} as a good reference. From the perspective of \textit{bases}, a $(r,n)$--\textit{matroid} is a nonempty subset $\sQ \subset \binom{[n]}{r}$ that satisfies the basis-exchange axiom. Here, the number $r$ is the \textit{rank} of $\sQ$, and $[n]$ is its ground set. We are primarily interested in matroids arising from matrices in the following way. Denote by $M_{r,n}(F)$ the set of full-rank $r\times n$ matrices with entries in the field $F$, and let $A\in M_{r,n}(F)$. Given a subset $I \in \binom{[n]}{r}$, denote by $A_{I}$ the $r\times r$ submatrix formed by the columns of $A$ indexed by $I$ (preserving their order).  The matroid of $A$ is 
\begin{equation}
\label{eq:matroidOfMatrix}
    \sQ(A) = \left\{I \in \textstyle{\binom{[n]}{r}} \; \bigg| \; \det A_{I} \neq 0 \right\}.
\end{equation}
A $(r,n)$--matroid $\sQ$ is \textit{realizable over $F$} if it is of the form $\sQ(A)$ for some $A\in M_{r,n}(F)$. In this case, the matrix $A$ is a \textit{realization} of $\sQ$. A \textit{binary matroid} is a matroid realizable over $\FF_{2}$. 

The assignment $A \mapsto \sQ(A)$ is constant on orbits of the action $\GL{r}{F} \curvearrowright M_{r,n}(F)$ by left multiplication, so $\sQ(A)$ depends only on the row span of $A$. The orbit space $\GL{r}{F} \backslash M_{r,n}(F)$ is known as the \textit{Grassmannian} $\Gr(r,n;F)$, and it is a projective algebraic variety that parameterizes the $r$-dimensional linear subspaces of $F^{n}$. Given $L\in \Gr(r,n;F)$, its matroid $\sQ(L)$ is $\sQ(A)$ where $L$ is the row span of $A$. \edit{The next proposition follows from \cite[Proposition~6.4.1]{Oxley}.}

\begin{proposition}
\label{prop:GrQeq1}
The assignment $L \mapsto \sQ(L)$ defines a bijection between $\Gr(r,n;\FF_2)$ and the set of all binary $(r,n)$--matroids.   
\end{proposition}

Denote by $\Gr^{\simp}(r,n;F) \subset \Gr(r,n;F)$ the \textit{distinct column} locus of the Grassmannian, i.e., the set of $r$-dimensional linear subspaces represented by matrices whose columns are distinct and not identically zero. This is an open subvariety of $\Gr(r,n;F)$, see Proposition \ref{prop:GrdcOpen}.  In view of Proposition~\ref{prop:GrQeq1}, the set $\Gr^{\simp}(r,n;\FF_2)$ may be identified with the set of matrices $A\in M_{r,n}(\FF_2)$ in row-reduced echelon form, whose columns are distinct and not identically zero. Since two vectors in $(\FF_2)^r$ are distinct if and only if they are not scalar multiples of each other, we see that $\Gr^{\simp}(r,n)$ is the set of  $r$-dimensional linear subspaces whose matroids are simple, i.e., have no loops or parallel elements. So Proposition \ref{prop:GrQeq1} also implies that $|\Gr^{\simp}(r,n;\FF_2)|$ equals the number of simple binary $(r,n)$--matroids. 

The automorphism group of a matroid $\sQ$ on $[n]$ is the subgroup of $\Sn{n}$ consisting of the permutations on $[n]$ that preserve the bases of $\sQ$.  Denote by $\pB(r,n)$ the set of isomorphism classes of simple binary $(r,n)$--matroids.

\begin{proposition}
\label{prop:BrnGrsimp}
We have
\begin{equation*}
    \sum_{\sQ\in \pB(r,n)} \frac{1}{|\Aut(\sQ)|} = \frac{\left|\Gr^{\simp}(r,n;\FF_2)\right|}{n!}.
\end{equation*}
\end{proposition}

\begin{proof}
Let $\pO_{\sQ} = \{\sQ' \, : \, \sQ'\cong \sQ\}$. By the orbit-stabilizer formula, we have $|\mathfrak{S}_n| = |\Aut(\sQ)||\pO_{\sQ}|.$
Rearranging this identity and summing over all matroids $\sQ\in \pB(r,n)$ gives
\begin{equation*}
    \sum_{\sQ\in \pB(r,n)} \frac{1}{|\Aut(\sQ)|} = 
    \frac{1}{n!} \sum_{\sQ\in \pB(r,n)} |\pO_{\sQ}|. 
 \end{equation*}
 The sum that appears on the right is exactly the number of simple binary $(r,n)$--matroids. As remarked above, this number is $|\Gr^{\simp}(r,n;\FF_2)|$ by Proposition \ref{prop:GrQeq1}.
\end{proof}

\section{Counting points in the distinct column locus of the Grassmannian}

 By Proposition \ref{prop:BrnGrsimp}, we have
 \begin{equation*}
    \chi(\pB(r)) = \sum_{n \geq r} \frac{(-1)^n}{n!} \left|\Gr^{\simp}(r,n;\FF_2)\right|.
\end{equation*}
In this section we prove Theorem \ref{thm:grsimp}, in which we compute the analog of the right-hand side of this equation where $\FF_{2}$ is replaced by any finite field $\FF_{q}$ where $q=p^e$ for any prime number $p$. First, we record three identities that we will need. 

\begin{lemma}\cite[Identity 24]{Spivey}\label{lem:24}
For any integer $n\geq0$:
\[
\sum_{k=0}^{n} \binom{n}{k}  \frac{(-1)^k}{k+1} = \frac{1}{n+1}.
\]
\end{lemma}

\begin{lemma}\cite[Identity 207]{Spivey}\label{lem:207}
For any integers $N,M,R\geq1$.
\[
\sum_{t = 0}^N \binom{N}{t} \frac{(-1)^t}{(M+t)(M+t+1) \cdots (M + R + t)}= \frac{(M-1)!(R+N)!}{(M+R+N)!R!}.
\]
\end{lemma}

\begin{lemma}
Fix a natural number $r\in \ZZ_{\geq 0}$. For any natural number $k$ such that $r \leq k \leq q^{r+1}-1$:
\begin{equation}
\label{eq:binary-lem}
\sum_{n=k}^{q^{r+1}-1} 
\frac{(-1)^n}{n!} \prod_{j = k+1}^{n} (q^{r+1} - j) 
= \frac{(-1)^k}{(q^{r+1}-1)(k-1)!}.
\end{equation}
\label{lem:weirdeqn}
\end{lemma}

\begin{proof}
The claimed identity follows from Lemma~\ref{lem:207} by multiplying by $(-1)^{k}$ and making the substitutions $M=1$, $R=k-1$, and $N=q^{r+1}-k-1$. In particular, multiplying the identity in Lemma~\ref{lem:207} by $(-1)^{k}$, making the substitution $M=1$, and simplifying gives
\begin{equation}\label{eq:lemma3.3-1}
(-1)^{k}\sum_{t=0}^{N}\binom{N}{t}\frac{(-1)^{t}\cdot t!}{(t+R+1)!}
=\frac{(-1)^{k}}{R+N+1)R!}.  
\end{equation}
Further, making the substitution $R=k-1$ and simplifying -- using the factorial definition of binomial coefficients on the left hand side -- \eqref{eq:lemma3.3-1}
becomes
\begin{equation}\label{eq:lemma3.3-2}
    \sum_{t=0}^{N}\frac{(-1)^{t+k}}{((t+k)!}\; \prod_{i=0}^{t-1}(N-i)
    =
    \frac{(-1)^{k}}{((N+k)(k-1)!}.
\end{equation}
From here the claimed identity in follows from equation \eqref{eq:lemma3.3-2} by re-indexing the product on the left hand side to start at $k+1$, re-indexing the sum on the left hand side to start at $k$, and making the final substitution $N=q^{r+1}-k-1$. 
\end{proof}

\begin{manualtheorem}{\ref{thm:grsimp}}
We have
\begin{equation*}
\sum_{n \geq r} \frac{(-1)^n}{n!} \left|\Gr^{\simp}(r,n;\FF_q)\right| = \prod_{i=1}^r\frac{1}{(1-q^i)}.
\end{equation*}
\end{manualtheorem}

\begin{proof}
We proceed by induction on $r$, starting with $r=1$. The elements of $\Gr^{\simp}(1,n;\FF_{q})$ are in bijection with ordered $(n-1)$--tuples of distinct elements of $\FF_q\setminus \{0,1\}$, and so
\begin{equation*}
    \frac{\left|\Gr^{\simp}(1,n;\FF_{q})\right|}{n!} = \frac{(q-2)!}{n!(q-n-1)!} = \frac{1}{n}\binom{q-2}{n-1}.
\end{equation*}
Summing over $n$ and applying \edit{the identity in Lemma~\ref{lem:24}} yields the theorem for $r=1$.

Then, it suffices to show that 
\begin{equation}
\label{eq:pf2eq1}
 \sum_{n = r+1}^{q^{r+1}-1} \frac{(-1)^n}{n!} \left|\Gr^{\simp}(r+1,n; \FF_q)\right| =\frac{-1}{q^{r+1}-1} \sum_{n = r}^{q^r-1} \frac{(-1)^n}{n!} \left|\Gr^{\simp}(r,n;\FF_q)\right|. 
\end{equation}

\noindent As a \edit{step} towards proving the identity in \eqref{eq:pf2eq1} we observe that
\begin{equation}
\label{eq:groth}
    \left|\Gr^{\simp}(r+1,n;\FF_q)\right| = 
\sum_{k=r+1}^{n} \left|\Gr^{\simp}(r,k-1;\FF_q)\right|   \prod_{j = k+1}^n (q^{r+1} - j). 
\end{equation}
\edit{When $k=n$, the product on the right is equal to 1.} To see this suppose $A\in M_{r+1,n}(\FF_q)$ is in row-reduced echelon form. Denote by $k$ the column index of $(0,\ldots,0,1)^t$. The upper-left $r\times (k-1)$ submatrix of $A$ represents an element of $\Gr^{\simp}(r,k-1;\FF_q)$, and there are  $\prod_{j = k+1}^n (q^{r+1} - j)$ 
many ways to fill in the remaining columns to the right of the $k$th column, and Equation \eqref{eq:groth} follows. See \S \ref{sec:K0} for a promotion of Equation \eqref{eq:groth} to the Grothendieck ring of varieties.

Using \eqref{eq:groth} we may rewrite the left hand side of Equation \eqref{eq:pf2eq1} as
\begin{equation*}
 \sum_{n = r+1}^{q^{r+1}-1} \frac{(-1)^n}{n!} \left|\Gr^{\simp}(r+1,n;\FF_q)\right|= \sum_{n = r+1}^{q^{r+1}-1} \frac{(-1)^n}{n!} \sum_{k=r+1}^{n} \left|\Gr^{\simp}(r,k-1;\FF_q)\right| \prod_{j = k+1}^n (q^{r+1} - j). 
\end{equation*}
Switching the indices and re-arranging terms, we have
\begin{equation}
\label{eq:pf2eq2}
     \sum_{n = r+1}^{q^{r+1}-1} \frac{(-1)^n}{n!} \left|\Gr^{\simp}(r+1,n;\FF_q)\right| =\sum_{k=r+1}^{q^{r}}
\left|\Gr^{\simp}(r,k-1;\FF_q)\right|
\sum_{n=k}^{q^{r+1}-1}
\frac{(-1)^n}{n!}  \prod_{j = k+1}^n (q^{r+1} - j). 
\end{equation}
Here, the upper bound on $k$ is $q^r$ because $\left|\Gr^{\simp}(r,k-1;\FF_q)\right| = 0$ for $k$ larger than $q^r$. 

Lemma \ref{lem:weirdeqn} computes the numerical sum in the right hand side of Equation \ref{eq:pf2eq2}, giving

\begin{equation*}
\sum_{n = r+1}^{q^{r+1}-1} \frac{(-1)^n}{n!} \left|\Gr^{\simp}(r+1,n; \FF_q)\right| 
=\sum_{k=r+1}^{q^{r}}
\left|\Gr^{\simp}(r,k-1;\FF_q)\right|
\frac{(-1)^k}{(q^{r+1}-1)(k-1)!}. 
\end{equation*}
Re-indexing, we find
\begin{equation*}
\sum_{n = r+1}^{q^{r+1}-1} \frac{(-1)^n}{n!} \left|\Gr^{\simp}(r+1,n; \FF_q)\right|  =\sum_{k=r}^{q^{r}-1}
\frac{(-1)^{k+1}}{(q^{r+1}-1)}
\frac{\left|\Gr^{\simp}(r,k;\FF_q)\right|}{k!}
\end{equation*}
as desired.
\end{proof}

\begin{remark}
This result proves Theorem \ref{thm:binaryOEC} by setting $q=2$ and applying Proposition~\ref{prop:BrnGrsimp}. 
\end{remark}

\begin{remark}
\label{sec:extendToOtherPrimesP}
Given that Theorem \ref{thm:grsimp} holds for all primes $p$, it is natural to wonder if we can compute analogous virtual Euler characteristics for arbitrary primes $p$. To be more precise, recall that a \emph{$p$-matroid} is a matroid that is realizable over $\FF_p$. Let $\pB_p(r)$ be the set of isomorphism classes of \textit{simple}  $p$-matroids of rank $r$, and $\pB_p(r,n)\subset \pB(r)$ those on the ground set $[n] = \{1,\ldots,n\}$. Define the \textit{$p$-matroid virtual Euler characteristic} by
\begin{equation*}
    \chi(\pB_p(r)) =  \sum_{n\geq r} \sum_{\sQ\in \pB_p(r,n)} \frac{(-1)^{n}}{|\Aut(\sQ)|}.
\end{equation*}
The main obstruction to proving an analogue of Theorem \ref{thm:binaryOEC} using our methods is that Proposition \ref{prop:BrnGrsimp} fails. Consider $p=3$.  There are four linear subspaces of  $\Gr(2,3;\FF_3)$ that realize the uniform $(2,3)$--matroid which are given by the row spans of the matrices:
\begin{equation*}
\begin{bmatrix}
1 & 0 & 1 \\
0 & 1 & 1
\end{bmatrix},\
\begin{bmatrix}
1 & 0 & -1 \\
0 & 1 & 1
\end{bmatrix},\
\begin{bmatrix}
1 & 0 & 1 \\
0 & 1 & -1
\end{bmatrix},\
\begin{bmatrix}
1 & 0 & -1 \\
0 & 1 & -1
\end{bmatrix},\
\end{equation*}
whereas Proposition \ref{prop:BrnGrsimp} says that over $\mathbb{F}_2$ there is only one. Thus, to get a formula for $\chi(\pB_p(r))$, we need to know the number of points the \textit{matroid stratum} for each $\sQ$, i.e., the points in $\Gr(r,n;\FF_p)$ realizing  $\sQ$. See \cite{Glynn, IampolskaiaSkorobogatovSorokin} for some results in this direction. 
\end{remark}

\section{Connections to the Grothendieck ring of varieties} \label{sec:K0}

In this section we show that the formula in Equation \eqref{eq:groth} has a promotion to an equality in the Grothendieck ring of $F$-varieties. By a \textit{$F$-variety} we mean a reduced separated finite type $F$-scheme. The Grothendieck ring of varieties over $F$, denoted by $K_0(\mathrm{Var}_{F})$, is the quotient of the free abelian group generated by isomorphism classes of $F$-varieties by relations of the form $[X] = [Z] + [X\setminus Z]$ where $Z\subset X$ is a closed subvariety. The product on $K_0(\mathrm{Var}_{F})$ is given by $[X] \cdot [Y] = [X\times_{\Spec(F)} Y]$. Thus $K_0(\mathrm{Var}_{F})$ is a commutative ring with multiplicative identity $[\Spec(F)]$. The Grassmannian $\Gr(r,n;F)$ is a closed subvariety of $\PP(\wedge^{r}F^n)\cong \PP(F^{\binom{n}{r}})$ defined by the vanishing of the Pl\"ucker ideal. For a standard reference, see \cite[Chapter~9]{Fulton:YT}. 
We may consider the class of $\Gr^{\simp}(r,n)$ in $K_0(\mathrm{Var}_{F})$ by the following proposition. 
\begin{proposition}
\label{prop:GrdcOpen}
    The subset $\Gr^{\simp}(r,n;F) \subset \Gr(r,n;F)$ is an open subvariety. 
\end{proposition}
We use Pl\"ucker coordinates to prove this.  Fix $L\in \Gr(r,n;F)$ given by the row-span of a matrix $A\in M_{r,n}(F)$. Given a length $r$ sequence $I$ of elements in $[n]$, denote by $A_{I}$ the submatrix of $A$ formed by the columns indexed by $I$ in the specified order.  The homogeneous Pl\"ucker coordinates of $L$ are: 
\begin{equation*}
    p_{I}(L) \coloneqq \det \, A_{I} \quad \quad \text{for} \quad \quad I \in \textstyle{\binom{[n]}{r}}.
\end{equation*}
These coordinates are well-defined up to multiplication by an element of $F\setminus \{ 0 \}$, and $p_{I}(A)$ is alternating in $I$. If $I$ is a sequence and $j\in [n]$, denote by $Ij$ the sequence obtained by appending $j$ to $I$.

Given a full-rank $r\times n$ matrix $A$,
the matroid $\sQ(A)$ of a matrix as defined in Equation \eqref{eq:matroidOfMatrix} admits the following characterization, which we also use in the proof of Proposition \ref{prop:GrdcOpen}. The matroid $\sQ(A)$ is the matroid of the vector configuration given by the columns of $A$ viewed as vectors in $F^{r}$, as defined in \cite[p.8]{Oxley}. That is, the independent sets are the subsets of column vectors that are linearly independent.

\begin{proof}[Proof of Proposition \ref{prop:GrdcOpen}]
Given indices $i,j\in [n]$ with $i\neq j$, set
\begin{align*}
         S_{i} &= \left \{L\in \Gr(r,n;F) \; \bigg| \; p_{J}(L) = 0 \text{ for all } |J|=r, \text{ and } i\in J \right \},\\
        S_{ij} & = \left\{L\in \Gr(r,n;F) \; \bigg| \; 
        \begin{array}{l}
        p_{J}(L) = 0 \text{ for all } |J|=r \text{ with } i,j\in J,  \text{ and } \\
        p_{Ii}(L) = p_{Ij}(L) \text{ for all } |I|=r-1, \text{ with } i,j\notin I
        \end{array}
         \right\} 
    \end{align*}
For a matrix $A$, denote by $A_{i}$ the $i$th column of $A$. We claim that 
\begin{equation*}
       S_{i} = \left\{L = \rowspan(A) \; \bigg| \; A_{i} = 0\right\} \hspace{20pt} \text{and} \hspace{20pt}  S_{ij} = \left \{L = \rowspan(A) \; \bigg| \; A_{i} = A_{j} \right\}.
    \end{equation*}
    Let $L=\rowspan(A)$. First consider $S_i$ and suppose $A_{i}\neq 0$. This means that $\{i\}$ is an independent set of $\sQ(A)$, which is contained in a basis $J$ of $\sQ(A)$ \cite[Theorem~1.2.3]{Oxley}. As $p_{J}(L) \neq 0$, we have that $L\notin S_{i}$. The converse is clear. 
    
    Now  consider $S_{ij}$, and suppose $A_{i} \neq A_{j}$. If $A_{i}$ and $A_{j}$ are linearly independent, then there is a basis $J$ of $\sQ(A)$ containing $i$ and $j$ [loc.\.cit.]. This means that $p_{J}(L) \neq 0$, i.e., $L\notin S_{ij}$. So assume that $A_{i}$ and $A_{j}$ are linearly dependent, and without loss of generality, say $A_{i} \neq 0$. In terms of $\sQ(A)$, this means that $\{i,j\}$ is a dependent set and $i$ is not a loop. Therefore, every basis $J$ of $\sQ(A)$ containing $i$ does not contain $j$. Fix such a basis $J$ (which exists since $i$ is not a loop), and let $A' = A_{J}^{-1}A$, so  $A'_J = I_r$. Note that $A'_{i}$ and $A_{j}'$ are also linearly dependent and not equal. The column vector  $A'_{i}$ is a standard basis vector, say $e_{k}$, and so $A'_{j} = \lambda e_{k}$ for some $\lambda \neq 1$. Let $I = J\setminus i$. Since $p_{Ii}(L) = \pm 1$ and $p_{Ij}(L) = \pm \lambda$ (that is, both $+$ or both $-$), we have that $p_{Ii}(L) \neq p_{Ij}(L)$, and hence $L\notin S_{ij}$. As in the $S_{i}$ case, the converse is clear. 
\end{proof}

 We promote Equation \eqref{eq:groth} to a recursive equality in $K_0(\mathrm{Var}_{F})$ that relates $[\Gr^{\simp}(r+1, n; F)]$ to  $[\Gr^{\simp}(r,k-1; F)]$ and $[\Conf_{n-k}(\A_{F}^{r+1} \setminus \{k+1 \text{ points}\})]$. Before stating the theorem, let us explain what is meant this last space. 
Given a variety $X$ denote by $\Conf_{m}(X)$ the configuration space of \textit{ordered} $m$ distinct points in $X$. The space $\Conf_{m}(X)$ may be understood as the open subvariety of $X^{m}$ given by
\begin{equation*}
    \Conf_{m}(X) = X^{m} \setminus \{(x_1,\ldots, x_m) \ :\ x_i = x_j \text{ whenever } i\neq j \}.
\end{equation*}
For $r\geq 1$ and an infinite field $F$, group $\Aut(\A_{F}^{r+1})$ acts $n$-transitively  $\A_{F}^{r+1}$ for each $n\geq 1$; we comment on this in the proof of the following theorem and a provide a reference. In particular, this means that $\A_{F}^{r+1} \setminus S$ is isomorphic to $\A_{F}^{r+1} \setminus T$ whenever $S$ and $T$ are finite sets of $F$-points of the same cardinality. This means that the class $[\Conf_{n-k}(\A_{F}^{r+1} \setminus \{k+1 \text{ points}\})]$ in the Grothendieck ring is well defined.

\begin{manualtheorem}{\ref{thm:K0}}
For integers $r\geq 1$ and $n\geq r+1$, and an infinite field $F$, we have
\begin{equation}
\label{eq:K0}
    [\Gr^{\simp}(r+1, n; \edit{F})] = \sum_{k=r+1}^{n} [\Gr^{\simp}(r,k-1; \edit{F})] [\Conf_{n-k}(\A_{\edit{F}}^{r+1} \setminus \{k+1 \text{ points}\})].
\end{equation}
\end{manualtheorem}

\begin{proof}
For $r+1 \leq k \leq n$, define 
\begin{equation*}
    Z_{k} = \left\{L \in \Gr^{\simp}(r+1,n;\edit{F}) \; \bigg| \; p_{I}(L) = 0 \text{ for all } I \in \textstyle{\binom{[k-1]}{r+1}}\right\}.
\end{equation*}
The $Z_{k}$'s are closed subsets of $\Gr^{\simp}(r+1,n;\edit{F})$ and  $\Gr^{\simp}(r+1,n;\edit{F}) = Z_{r+1} \supset Z_{r+2} \supset \cdots Z_{n-1} \supset Z_{n} = \emptyset$. Set $Y_{k} = Z_{k} \setminus Z_{k+1}$. In terms of matrices, given $L = \rowspan(A)$ in $\Gr^{\simp}(r+1,n;\edit{F})$  where $A$ is in row-reduced echelon form, $L \in Z_{k}$ if the $(r+1)$st pivot of $A$ occurs in column $\geq k$, whereas $L\in Y_{k}$ if the $(r+1)$st pivot of $A$ occurs in column $k$. 

Suppose $L\in Y_{k}$. There is a unique matrix $M(L)\in M_{r+1,n}(F)$ in row-reduced echelon form such that $\rowspan(M(L))=L$, with its $k$th row is the vector $(0,\ldots,0,1)^t$. Therefore, the row-span of the upper-left $r\times (k-1)$ matrix is an element of $\Gr^{\simp}(r, k-1)$. Using this, define a morphism $\pi:Y_{k} \to \Gr^{\simp}(r,k-1;F)$ by $\pi(L)\coloneqq \rowspan(M'(L))$ where $M'(L)$ is the upper-left $r\times (k-1)$ sub-matrix of $M(L)$. One readily checks that this map is regular. We now show that $\pi$ is a fibration in the Zariski topology with fiber isomorphic to $\Conf_{n-k}(\A_{F}^{r+1}\setminus T)$ where $|T|=k+1$, i.e., for all $L' \in \Gr^{\simp}(r,k-1;F)$ there exists a Zariski open set $L'\in U \subset \Gr^{\simp}(r,k-1;F)$ such that $\pi|_{\pi^{-1}(U)}:\pi^{-1}(U)\to U$ is isomorphic to the projection $U \times \Conf_{n-k}(\A^{r+1}_{F}\setminus T) \to U$.

Given $I = \{i_1,\ldots, i_r\}\subset [k-1]$, consider the open affine set
\begin{equation*}
    U_{I} = \Gr^{\simp}(r, k-1; F) \cap \{p_I\neq 0\}.
\end{equation*}
The variety structure of $U_I$ and $\pi^{-1}(U_I)$ may be described using affine coordinates for the Grassmannian which we now describe. See  \cite[Chapter~3]{GKZ} for a description on the passage between Pl\"ucker and affine coordinates. 
To simplify the notation, we consider only the case $I = \{1,\ldots,r\}$. Consider the $(r+1) \times n$  matrix of variables
\begin{equation*}
    A = \left[\begin{array}{ccccccc|cccc}
        1 & 0 & \cdots & 0 &  x_{1,r+1} & \cdots & x_{1,k-1} & 0 & y_{1,k+1} & \cdots & y_{1,n} \\
        \cdots & \cdots & \cdots & \cdots & \cdots & \cdots & \cdots & \cdots & \cdots & \cdots & \cdots \\
        0 & 0 & \cdots & 1 & x_{r,r+1} & \cdots & x_{r,k-1} & 0 & y_{r,k+1} & \cdots & y_{r,n}\\
        0 & 0 & \cdots & 0 & 0 & \cdots & 0 & 1 & y_{r+1,k+1} & \cdots & y_{r+1,n}\\
    \end{array}
    \right]
\end{equation*}
Denote by 
\begin{align*}
&\mathbf{x}_0 = \mathbf{0}, \\
&\mathbf{x}_1 = \mathbf{e}_1,\ \ldots,\ \mathbf{x}_r = \mathbf{e}_r, \quad  \\ 
&\mathbf{x}_{r+1} = (x_{1,r+1},\, \ldots,\, x_{r,r+1},0)^t,\ \ldots,\ \mathbf{x}_{k-1} = (x_{1,k-1},\, \ldots,\, x_{r,k-1},0)^t, \\
&\mathbf{x}_{k} = \mathbf{e}_{r+1}. 
\end{align*}
where $\mathbf{e}_1,\ldots, \mathbf{e}_{r+1}$ are the standard basis vectors. In the following, we write the coordinates of $\mathbf{x}_{\ell}$ as $(x_{1,\ell}, \ldots, x_{r+1,\ell})^{t}$ with the understanding that some of the $x_{ij}$ are $0$ or $1$ as defined above, and we write $\mathbf{x}$ for the tuple $(\mathbf{x}_0, \ldots, \mathbf{x}_{k})$, which we view as an element of $\A_{F}^{r\times (k-r-1)}$ by just picking out the coordinates that are variable. 
Then $U_{I}$ may be identified with the Zariski open subvariety of $\A_{F}^{r\times (k-r-1)}$ where the vectors $\mathbf{x}_0, \ldots, \mathbf{x}_{k-1}$ are distinct. 
Similarly, set 
\begin{equation*}
    \mathbf{y}_{k+1} = (y_{1,k+1},\ldots, y_{r,k+1},y_{r+1,k+1})^t, \ldots, \mathbf{y}_{k+1} = (y_{1,n},\ldots, y_{r,n},y_{r+1,n})^t.
\end{equation*}
Then $\pi^{-1}(U_I)$ is the Zariski open subvariety of $\A_{F}^{r\times (k-r-1)} \times \A_{F}^{(r+1) \times (n-k)}$ where the vectors $\mathbf{x}_0, \ldots, \mathbf{x}_{k}$, $\mathbf{y}_{k+1}, \ldots, \mathbf{y}_{n}$ are distinct.  

We describe, using Lagrange interpolation, a way of providing an automorphism of $\A_{F}^{r+1}$ that brings $\mathbf{x}_0,\ldots, \mathbf{x}_{k}$ to a fixed collection $T$ of $(k+1)$ distinct points in $\A_{F}^{r+1}$.  This provides a way of producing a family of isomorphisms of $\Conf_{n-k}(\A_{F}^{r+1}\setminus \{\mathbf{x}_0, \ldots, \mathbf{x}_{k}\})$ with $\Conf_{n-k}(\A_{F}^{r+1}\setminus T)$ that depend algebraically the $x$-variables. Our argument is inspired by that in \cite{Arzhantsev}, but we must be more explicit in demonstrating this last point. 

Assume that $T = \{\mathbf{c}_{0},\ldots, \mathbf{c}_{k+1}\}$ where $\mathbf{c}_{\ell} = (c_{1,\ell}, \ldots, c_{r+1,\ell})^t$ and assume that all entries are distinct. Let $h\colon \A_{F}^{r+1}\to \A_{F}^{1}$ be a linear form 
\begin{equation*}
    h(z_1,\ldots,z_{r+1}) = z_1 + a_2 z_2 + \cdots + a_{r+1} z_{r+1}
\end{equation*}
The coefficient of $z_1$ being $1$ ensures that the map $\Phi_3$ below is invertible.  
Consider the Zariski-open subvariety $\widetilde{U}_{I,h}\subset U_{I}$ where the values $h(\mathbf{x}_i)$ are pairwise distinct; $U_{I}$ is covered by these open sets.   Define functions
\begin{equation*}
    f(\mathbf{x}, z) = \sum_{\beta=0}^{k} (c_{1,\beta} - h(\mathbf{x}_{\beta})) \prod_{\gamma \neq \beta} \frac{z-c_{2,\gamma}}{c_{2,\beta} - c_{2,\gamma}},
    \qquad 
    g_i(\mathbf{x}, z) = \sum_{\beta=0}^{k}(c_{i,\beta} - x_{\beta}) \prod_{\gamma \neq \beta} \frac{z-h(\mathbf{x}_{\gamma})}{h(\mathbf{x}_{\beta}) - h(\mathbf{x}_{\gamma})}.
\end{equation*}
These functions are defined so that
\begin{equation*}
    f(\mathbf{x}, c_{2,\ell}) = c_{1,\ell} - h(\mathbf{x}_{\ell}), \quad \text{and} \quad g_{i}(\mathbf{x}, h(\mathbf{x}_{\ell})) = c_{i,\ell} - x_{i,\ell}.
\end{equation*}

Set 
\begin{align*}
    &\Phi_1 \colon \A_{F}^{r\times(k-r-1)} \times \A_{F}^{r+1} \to \A_{F}^{r\times(k-r-1)} \times  \A_{F}^{r+1}, \\
    &\Phi_1(\mathbf{x}; z_1,\ldots,z_{r+1}) = (\mathbf{x}; z_1 + f(\mathbf{x}, z_2), z_2, \ldots, z_{r+1}),
\end{align*}
and
\begin{align*}
    &\Phi_2\colon \A_{F}^{r\times(k-r-1)} \times \A_{F}^{r+1} \to \A_{F}^{r\times(k-r-1)} \times \A_{F}^{r+1}, \\
    &\Phi_2(\mathbf{x}; z_1,\ldots,z_{r+1}) = (\mathbf{x}; z_1, z_2 + g_2(\mathbf{x}, z_1), \ldots, z_{r+1} + g_{r+1}(\mathbf{x}, z_1)),
\end{align*}
and
\begin{align*}
    &\Phi_3 \colon \A_{F}^{r\times(k-r-1)} \times \A_{F}^{r+1} \to \A_{F}^{r\times(k-r-1)} \times \A_{F}^{r+1}, \\
    &\Phi_3(\mathbf{x}; z_1,\ldots,z_{r+1}) = (\mathbf{x}; h(z_1,\ldots,z_{r+1}), z_2, \ldots, z_{r+1}).
\end{align*}
Observe that $\Phi_1$ and $\Phi_3$ are both automorphisms of $\A_{F}^{r\times(k-r-1)} \times \A_{F}^{r+1}$; their inverses are 
\begin{align*}
    &\Phi_{1}^{-1}(\mathbf{x}; w_1,\ldots,w_{r+1}) = (\mathbf{x}; w_1 - f(\mathbf{x}, w_2), w_2, \ldots, w_{r+1}) \\
    &\Phi_{3}^{-1}(\mathbf{x}; w_1,\ldots,w_{r+1}) = (\mathbf{x}; w_1 - a_2 w_2 - \cdots - a_{r+1} w_{r+1}, w_2, \ldots, w_{r+1})
\end{align*}
The map $\Phi_{2}$ is an automorphism on the Zariski open set $\widetilde{U}_{I,h} \times \A_{F}^{r+1}$ and an inverse is given by
\begin{equation*}
    \Phi_2^{-1}(\mathbf{x}; w_1,\ldots,w_{r+1}) = (\mathbf{x}; w_1, w_2 - g_2(\mathbf{x}, w_1), \ldots, w_{r+1} - g_{r+1}(\mathbf{x}, w_1))
\end{equation*}

Finally, let $\pi_{2}\colon \A_{F}^{r\times(k-r-1)} \times \A_{F}^{r+1} \to \A_{F}^{r+1}$ be the projection onto the second factor. 
Define $\Theta = \pi_{2}\circ\Phi_1\circ \Phi_2 \circ \Phi_3$. For any fixed $\mathbf{x}\in \widetilde{U}_{I,h}$, we have that $\theta_{\mathbf{x}} = \Theta(\mathbf{x}; -)$ is an automorphism of $\A_{F}^{r+1}$ that sends $\mathbf{x}_{\ell}$ to $\mathbf{c}_{\ell}$ for $\ell = 0,\ldots,k$. 

The map
\begin{align*}
    &\Psi \colon \A_{F}^{r\times (k-r-1)} \times \A_{F}^{(r+1) \times (n-k)} \to \A_{F}^{r\times (k-r-1)} \times \mathsf{Conf}_{n-k}(\A_{F}^{(r+1)}\setminus T) \\
    &\Psi(\mathbf{x}; \mathbf{y}_{k+1},\ldots, \mathbf{y}_{n}) = (\mathbf{x}; \theta_{\mathbf{x}}( \mathbf{y}_{k+1}), \ldots, \theta_{\mathbf{x}}(\mathbf{y}_{n}))
\end{align*}
restricts to an isomorphism $\pi^{-1}(\widetilde{U}_{I,h}) \to \widetilde{U}_{I,h} \times \Conf_{n-k}(\A_{F}^{(r+1)}\setminus T)$. Indeed, an inverse is given by
\begin{equation*}
    \Theta^{-1}(\mathbf{x}; \mathbf{w}_{k+1}, \ldots, \mathbf{w}_{n}) = (\mathbf{x}; \theta_{\mathbf{x}}^{-1}(\mathbf{w}_{k+1}), \ldots, \theta_{\mathbf{x}}^{-1}(\mathbf{w}_{n}) )
\end{equation*}

It is well known that if $E\to B$ is a fibration in the Zariski topology with fiber $F$, then $[E] = [B][F]$ in the Grothendieck ring of varieties. We just showed that $\pi:Y_{k} \to \Gr^{\simp}(r,k-1;F)$ is a fibration in the Zariski topology with fiber $\Conf_{n-k}(\A_{F}^{r+1} \setminus \{k+1 \text{ points}\})$, so we have
\begin{align*}
    [\Gr^{\simp}(r+1, n; \edit{F})] = 
    \sum_{k=r+1}^{n} [Y_{k}]
    = \sum_{k=r+1}^{n} [\Gr^{\simp}(r,k-1; \edit{F})] [\Conf_{n-k}(\A_{\edit{F}}^{r+1} \setminus \{k+1 \text{ points}\})]. 
\end{align*}
\end{proof}

\bibliographystyle{abbrv}
\bibliography{sample}
\label{sec:biblio}

\end{document}